\documentclass[11pt]{amsart}
\usepackage[margin=1in]{geometry}
\usepackage{yufei}
\numberwithin{equation}{section}
\newcommand{\lcm}{\text{lcm}}
\title[Classification of numbers satisfying a theta function identity]{Proof of the Ballantine-Merca Conjecture and theta function identities modulo $2$}
\date{January 2021}
\author{Letong Hong, Shengtong Zhang}
\address{Department of Mathematics, Massachusetts Institute of Technology, Cambridge, MA 02139}
\email{clhong@mit.edu, stzh1555@mit.edu}

\begin{document}
\begin{abstract}
\vspace{-1em}
For positive integers $m$ we consider the theta functions $f_m(z):=\sum_{mk+1\text{ square }}q^k$. Due to classical identities of Jacobi, it is known that $$f_4\equiv f_6f_{12}\pmod 2.$$ Here we prove that the only triples $(a,b,c)$ for which $f_a\equiv f_bf_c\pmod 2$ are of the form $(2q,4q,4q)$ or $(4q,4q,8q)$, where $q$ is any positive odd number, or belong to the following finite list $$\{(4,6,12),(6,8,24),(8,12,24),(10,12,60),(15,24,40),(16,24,48),(20,24,120),(21,24,168)\}.$$ The result is inspired by the Ballantine-Merca Conjecture on recurrence relations for the parity of the partition function $p(n)$, which we also prove here.
\vspace{-1em}
\end{abstract}
\maketitle
\section{Introduction and Statement of Results}
 We recall the famous theta functions defined by Euler and Jacobi:
\begin{equation}\label{eulerjacobi}
(q;q)_{\infty} = \sum_{n = -\infty}^\infty (-1)^n q^{n(3n + 1) / 2}, \end{equation}\vspace{-0.5em}
\begin{equation*}
(q;q)_{\infty}^3 = \sum_{n = 0}^\infty (-1)^n (2n + 1) q^{n(n + 1) / 2},\end{equation*}
 where $(q;q)_{\infty}:=\prod_{n=1}^{\infty} (1-q^n).$
These identities motivate the definition of the following class of theta functions:
$$f_a(q) := \sum_{an + 1 \text{ is a square}} q^n.$$
Observe that Euler and Jacobi's identities give us
$$(q;q)_{\infty}^a \equiv f_{24/a}(q) \pmod 2$$
for $a \in \{1,2,3,4,6\}$. The goal of our paper is to classify such identities modulo $2$. Our main result is the following classification theorem.
\begin{theorem}
\label{thm:Classification}
The only triples of positive integers $(a,b,c)$ such that $f_a\equiv f_bf_c\pmod 2$ are of the form $$(2q,4q,4q),(4q,8q,8q),$$ where $q$ is any positive odd number, or are members of the finite set: $$\{(4,6,12),(6,8,24),(8,12,24),(10,12,60),(15,24,40),(16,24,48),(20,24,120),(21,24,168)\}.$$
\end{theorem}
This theorem is related to recent work of Ballantine and Merca \cite{BM17} on the integer partition function. Recall that a partition of any non-negative integer $n$ is any non-increasing sequence of positive integers which sum to $n$, and $p(n)$ denotes their total number (by convention, $p(0) = 1$). The generating function for $p(n)$ is
$$P(q):=\sum_{n = 1}^\infty p(n)q^n = \prod_{n = 1}^\infty \frac{1}{1 - q^n} = \frac{1}{(q;q)_{\infty}}.$$
Thanks to the Euler-Jacobi identity (\ref{eulerjacobi}), this gives the recurrence
$$p(n) = \sum_{j = 1}^n (-1)^{j + 1}\left[p\left(n - \frac{1}{2}j(3j - 1)\right) + p\left(n - \frac{1}{2}j(3j + 1)\right)\right].$$
Ballantine and Merca demonstrated further recurrences for the parity of $p(n)$. They proved that for $(a,b)$ in a certain set (listed below), the quantity $$\sum_{ak+1 \text{ square}}p(n-k)$$ is odd if and only if $bn+1$ is a square. They asked whether the converse is true.
\begin{conjecture*}[\cite{BM17}]
We have that 
\begin{equation}
\label{eq: main}
 \sum_{\text{ak+1 is square}}p(n-k)\equiv 1\pmod 2 \ \ \text{ if and only if } bn+1 \text{ is a square}.   
\end{equation} is true only for $(a,b)\in \{(6, 8), (8, 12), (12, 24), (15, 40), (16, 48), (20, 120), (21, 168)\}.$
\end{conjecture*}
We note that statement \cref{eq: main} is equivalent to the modulo $2$ theta function identity $$f_a \equiv f_bf_{24} \pmod 2.$$ Thus \cref{thm:Classification} resolves Ballantine-Merca's conjecture.
\begin{corollary}
\label{thm:B-M}
The Ballantine-Merca Conjecture is true.
\end{corollary}
The paper is organized as follows. In Section \ref{section2} we give an important proposition that leads directly to the proof of the Ballantine-Merca Conjecture, and in Section \ref{section3} we prove Theorem \ref{thm:Classification} by discovering some further constraints.

\section*{Acknowledgements}
We would like to thank Professor Ken Ono for his guidance and many insightful discussions. We also thank Dr Mircea Merca for pointing out a typographical error in an earlier version of the paper. We are grateful to the anonymous referee for detailed and helpful suggestions. The research was supported by the generosity of the National Science Foundation under grant DMS-2002265, the National Security Agency under grant H98230-20-1-0012, the Templeton World Charity Foundation, and the Thomas Jefferson Fund at the University of Virginia.

\section{Proof of \cref{thm:B-M}}\label{section2}
In this section we prove the following important proposition that resolves \cref{thm:B-M}. 
\begin{proposition}
\label{prop:main}
If $a,b,c$ are positive integers with $f_a\equiv f_bf_c\pmod 2$, then $\frac{1}{a}=\frac{1}{b}+\frac{1}{c}.$
\end{proposition}
\begin{proof}
If $f_a\equiv f_bf_c\pmod 2$ then for all $k_1\in \mathbb{N}$, $ak_1+1=x^2$ for some integer $x$ if and only if there exists an odd number of pairs $(y,z)\in \mathbb{N}^2$ such that \begin{equation}\label{eq0.1}
k_1 =\frac{y^2-1}{b}+\frac{z^2-1}{c},\quad b \mid (y^2 - 1),\quad c \mid (z^2 - 1).\end{equation} We rewrite the equation as \begin{equation}\label{eq0.2} bz^2+cy^2=bck_1 + b + c,\quad b \mid (y^2 - 1),\quad c \mid (z^2 - 1).\end{equation}
Notice that if we take a prime number $p$ such that the Legendre symbol $( \frac{-bc}{p})=-1$, the integer solutions for the equation $bm^2+cn^2=N$ are in bijection with the integer solutions for $bm^2+cn^2=p^{2t}N$ by considering $(m,n) \to (p^t m, p^t n)$. Thus, if we take $t$ such that $p^t \equiv 1\pmod{bc}$, then the number of positive integer solutions to 
$$bz^2+cy^2=bck_1 + b + c,\quad b \mid (y^2 - 1),\quad c \mid (z^2 - 1).$$
is the same as the number of positive integer solutions to
$$bz^2+cy^2=p^{2t}(bck_1 + b + c),\quad b \mid (y^2 - 1),\quad c \mid (z^2 - 1).$$
Now note that
$$p^{2t}(bck_1 + b + c) = bc\left(p^{2t}k_1 + \frac{(p^{2t} - 1)(b + c)}{bc}\right) + b + c.$$
Therefore, we conclude that $ak_1 + 1$ is a perfect square if and only if
$$a\left(p^{2t}k_1 + \frac{(p^{2t} - 1)(b + c)}{bc}\right) + 1$$
is a perfect square. There is an infinite number of $k_1$ such that $ak_1 + 1$ is square. For each such $k_1$, we have
$$a\left(p^{2t}k_1 + \frac{(p^{2t} - 1)(b + c)}{bc}\right) + 1 = p^{2t}(ak_1 + 1) + (p^{2t} - 1)\left(b^{-1} + c^{-1} - a^{-1}\right).$$
Thus the perfect square
$$a\left(p^{2t}k_1 + \frac{(p^{2t} - 1)(b + c)}{bc}\right) + 1$$
and the other perfect square
$$p^{2t}(ak_1 + 1)$$
differ by a constant
$$(p^{2t} - 1)\left(b^{-1} + c^{-1} - a^{-1}\right).$$ For infinitely many $k_1$ to satisfy this property above, the only possibility is that 
$$(p^{2t} - 1)\left(b^{-1} + c^{-1} - a^{-1}\right) = 0,$$
which implies 
$1/a = 1/b + 1/c,$ the desired result.
\end{proof}

As we stated in the introduction, the modulo constraint \cref{eq: main} in Ballantine-Merca Conjecture is true if and only if $$f_a\equiv f_bf_{24}\pmod 2.$$ Applying \cref{prop:main} we obtain $$\frac{1}{a}=\frac{1}{b}+\frac{1}{24}$$ and the listed pairs in Ballantine and Merca's paper plus $(22, 264), (23, 552)$ are all solutions. Since they already verified all $a,b \leq 100000$, the conjecture is solved.

\section{Proof of \cref{thm:Classification}} \label{section3}
In this section we apply more sophisticated analysis to prove \cref{thm:Classification}. Throughout the section, we assume $a,b,c$ are positive integers with $f_a \equiv f_bf_c\pmod 2$. We denote $d = \gcd(b,c)$, $b' = \frac{b}{d}$, and $c' = \frac{c}{d}$. For a prime $p$ and integer $n$, denote $v_p(n) = \max \{a \geq 0: p^a | n\}.$
\subsection{A bound on $b' + c'$.}
\begin{proposition}
\label{prop:b'c'}
If positive integers $a,b,c$ satisfy $f_a \equiv f_bf_c\pmod 2$, then we must have $$(b' + c') \mid 24.$$
Furthermore, if $b' + c'$ is divisible by $2$ then we have $v_2(d) \leq 4$, and if $b' + c'$ is divisible by $3$ then we have $v_3(d) \leq 1$. 
\end{proposition}
\begin{proof}
By \cref{prop:main} we have $\frac{1}{a} = \frac{1}{b} + \frac{1}{c}$. Our condition $f_a\equiv f_bf_c\pmod 2$ corresponds to the statement that the number of solutions $(y,z)\in \mathbb{N}^2$ to\begin{equation}\label{eqabc} cy^2+bz^2=bcx+b+c=(b+c)(ax+1),\quad b\mid (y^2-1),\quad c\mid (z^2-1) \end{equation}
is odd if and only $ax+1$ is a square.

First note that $(b'+c')\mid d$. This is because $a(b'+c')=db'c'$ and $\gcd (b'c',b'+c')=1$.

We consider any prime number $p\neq 2,3$ satisfying the following modulo constraints
\begin{equation}
\label{assumptions}
a\mid (p^2-1),\quad \left( \frac{-b'c'}{p}\right)=-1.
\end{equation}
Letting $\delta = \frac{p^2 - 1}{a}$, we substitute $x=\delta $ to Equation (\ref{eqabc}) and obtain $cy^2+bz^2=p^2(b+c),$ or $c'y^2+b'z^2=p^2(b'+c')$. Since $\big( \frac{-b'c'}{p}\big)=-1,$ it follows $p$ divides $y,z$. Furthermore, any solution $(y, z)$ of \cref{eqabc} must have $y > 0$, as otherwise $b | y^2 - 1 = 1$, so $b = 1$, which contradicts the identity $\frac{1}{a} = \frac{1}{b} + \frac{1}{c}$. Symmetrically, we must have $z > 0$. Thus the only possible solution to \cref{eqabc} is $(y,z)=(p,p)$. Since $ax + 1 = a\delta + 1 = p^2$ is a square, \cref{eqabc} must have an odd number of solutions, thus $(y,z) = (p,p)$ must be a solution to \cref{eqabc}. Therefore, we obtain $b\mid (p^2-1),\ c\mid (p^2-1),$ or equivalently
\begin{equation}
\label{consequence}
db'\mid (p^2-1),\quad dc'\mid (p^2-1).
\end{equation}

We have thus shown that for any prime $p$, \cref{assumptions} implies \cref{consequence}. We now attempt to construct some $p$ that satisfies \cref{assumptions} but not \cref{consequence}. We next introduce a technical lemma to establish the existence of a desirable $p$.

\begin{lemma}\label{existence-of-p}
Assume $u$ and $v$ are positive integers with $v$ square free. Let $Q = \lcm(8, u, v)$. Then there exists a residue class $q$ mod $Q$ with $(q, Q) = 1$, such that for any prime number $p \equiv q \pmod Q$ we have 
$$p^2 - 1\equiv 0 \pmod u,\quad \left(\frac{-v}{p}\right) = -1.$$
Furthermore, if $v$ contains a prime divisor that is 3 mod 4 and $u$ is divisible by $8$, we could further take $q$ mod $Q$ such that $v_2(p^2 - 1) = v_2(Q)$.
\end{lemma}

Before showing the lemma, we see how it finishes the proof of \cref{prop:b'c'}. We take $u = a = \frac{d}{b' + c'}b'c'$ and $v$ be the square-free reduction of $b'c'$. For every prime $q\mid (b'+c')$, we write $b'+c'=q^{\alpha}s,\ d=q^{\beta}d'$ where $\gcd (s,q)=\gcd(d',q)=1$. As $b' + c' | d$, we have $\alpha \leq \beta$. Note that
$$v_q(u) = \beta - \alpha,\quad v_q(v) = 0.$$

If $q\geq 3$, then \cref{existence-of-p} asserts that fixing $p \pmod{q^{\beta - \alpha}}$ suffices to ensure that $p$ satisfies \eqref{assumptions}. But for all such $p$, we must have $q^{\beta}\mid (p^2-1)$ by \cref{consequence}. This implication holds for all $p$ only if $q = 3$ and $\beta = 1$. 

If $q = 2$, we show that $\beta \leq 3$. If $\beta \geq 4$, then both $b'$ and $c'$ are odd, and $b' + c' \equiv 0 \pmod{16}$, so $v$ must have a prime divisor of 3 mod 4. By \cref{existence-of-p}, we could take $p$ such that $v_2(p^2 - 1) = v_2(Q)$ and $p$ satisfies \eqref{assumptions}. Now
$$v_2(Q) = \max(\beta - \alpha, 3) < \beta$$
contradiction the consequence \cref{consequence} that $q^{\beta}\mid (p^2-1)$. Therefore we must have $\beta \leq 3$ and $\alpha \leq \beta \leq 3$, finishing the proof.
\end{proof}
\begin{proof}[Proof of Lemma \ref{existence-of-p}]
We assume the prime factorization of $v$ is $v=q_1q_2 \cdots q_m$. Let the prime factorization of $u$ be $u=p_1^{a_1}\cdots p_n^{a_n}q_1^{l_1}q_2^{l_2}\cdots q_m^{l_m}$, where $p_1,\cdots,p_n$ are prime factors distinct from $q_1,\cdots,q_m$. Then we want to satisfy the congruence constraints
\begin{equation}
 \label{eq: constraints}
p^2 \equiv 1 \pmod{p_i^{a_i}},\quad p^2 \equiv 1 \pmod{q_i^{l_i}},      
\end{equation}
and $\big(\frac{-q_1q_2\cdots q_m}{p} \big)=-1$. Now we do case work. In each case, we describe a number of congruence conditions for $p$, and show that they guarantee that $p$ satisfy the above constraints. Then the Chinese remainder theorem shows that a modulo class mod $Q$ can be taken such that all the congruence conditions are satisfied.

\textbf{Case 1}: If there exists an $1 \leq i \leq m$ such that $q_i$ is 3 mod 4. Then we could first satisfy \cref{eq: constraints} for all other prime divisors of $Q$ arbitrarily. In particular, we could choose $p \pmod{2^{v_2(Q)}}$ such that $v_2(p^2 - 1) = v_2(Q)$. Then by quadratic reciprocity,
$$\left(\frac{-q_1q_2\cdots q_m}{p} \right) = c \left(\frac{p}{q_i}\right),$$
where $c \in \{\pm 1\}$ is fixed. So we could adjust the sign of $p \pmod{q_i}$ to ensure that $p \equiv \pm 1 \pmod{q_i^{l_i}}$ and $\big(\frac{-q_1q_2\cdots q_m}{p} \big) = -1$. 

\textbf{Case 2}: If $v$ is odd and all $q_i$ are 1 mod 4. Then by quadratic reciprocity we have
$$\left(\frac{-q_{1}\cdots q_{m}}{p} \right)=\left(\frac{-1}{p}\right) \left(\frac{p}{q_{1}}\right)\cdots \left(\frac{p}{q_{m}}\right).$$
Now note that to ensure $p^2 \equiv 1 \pmod{2^{v_2(u)}}$, we could take either of $p \equiv \pm 1 \pmod{2^{v_2(Q)}}$. Thus we could first choose modulo classes that satisfy \cref{eq: constraints} for all odd prime divisors of $Q$, then choose the sign in $p \equiv \pm 1 \pmod{2^{v_2(Q)}}$ to ensure the quadratic character $\big(\frac{-q_1q_2\cdots q_m}{p} \big) = -1$.

\textbf{Case 3}: If $q_{1}=2$ and all other $q_i$ is 1 mod 4. We observe $$\left(\frac{-2q_{2}\cdots q_{m}}{p} \right)=\left(\frac{-2}{p}\right) \left(\frac{p}{q_{1}}\right)\cdots \left(\frac{p}{q_{m}}\right).$$ 
Now note that to ensure $p^2 \equiv 1 \pmod{2^{v_2(u)}}$, we could take either $p \equiv \pm 1 \pmod{2^{v_2(Q)}}$. If  $p \equiv 1 \pmod{2^{v_2(Q)}}$ then $\big(\frac{-2}{p}\big) = 1$, and if $p \equiv -1 \pmod{2^{v_2(Q)}}$ then $\big(\frac{-2}{p}\big) = -1$. Thus we could first choose modulo classes that satisfy \cref{eq: constraints} for all odd prime divisors of $Q$, then choose the $\pm 1$ to ensure the quadratic character is $-1$.
\end{proof}

\subsection{A bound on $d$.}
The final piece of our proof is the following proposition.
\begin{proposition}
\label{prop:d}
Suppose positive integers $a,b,c$ satisfy $f_a \equiv f_bf_c\pmod 2$ and $b' \neq c'$. Then $d \mid 24(b' - c')$.
\end{proposition}
We first prove a lemma.
\begin{lemma}
\label{lem: solutions}
Suppose $p$ is a prime coprime to $b'c'$ that can be expressed in the form $u^2 + b'c' v^2$, where $u,v$ are integers. Then the only integer solutions to $c'y^2 + b'z^2 = (b' + c')p$ with $y$ coprime to $b'$ and $z$ coprime to $c'$ are (up to sign) $(y,z) = (u - b'v, u + c'v)$ and $(y,z) = (u + b'v, u' - c'v)$.
\end{lemma}
\begin{proof}
If $c'y^2 + b'z^2 = (b' + c')p$, then by congruence modulo $p$ we must have 
$$yz^{-1} \equiv \pm b'vu^{-1}\pmod p.$$
By swapping the sign of $y,z$, we can, without loss of generality, assume that
$$yz^{-1} \equiv b'vu^{-1}\pmod p.$$
Then we have 
$$c'\left(\frac{uy - b'vz}{p}\right)^2 + b'\left(\frac{uz + c'vy}{p}\right)^2 = b' + c'.$$
Now that the quantities in the bracket are both integers, either both are $\pm 1$ or one of them is $0$. The latter is impossible by the imposed coprime condition, and the former gives the two pairs of solutions.
\end{proof}
\begin{proof}[Proof of \cref{prop:d}]
Let $q^{h}$ be any maximal power of a prime dividing $d$. It suffices to show $q^h \mid 24(b' - c')$. The case of $b' + c' $ not coprime to $q$ is handled in \cref{prop:b'c'}, so we assume $b' + c'$ is coprime to $q$. 

Assume $\lcm(a, b, c) = q^w r$, where $(r, q) = 1$. We consider the congruence classes $s, t$ modulo $rq^w$ defined by
$$s \equiv 1 \pmod r,\quad s\equiv \frac{b' - c'}{b' + c'}\pmod{q^w},$$
and
$$t \equiv 0 \pmod r,\quad t\equiv \frac{2}{b' + c'} \pmod{q^w},$$
which are well-defined as $(b' + c', q) = 1$.
Then we have
$$s^2 + b'c' t^2 \equiv 1 \pmod{rq^w}.$$
Therefore, by a well-known theorem of Weber(\cite{Weber}), there exist infinitely many primes $p$ such that $p = u^2 + b'c' v^2$ with $u \equiv s \pmod{rq^w} $, $v \equiv t \pmod{rq^w} $. Furthermore, for all such $p$ we have $p \equiv 1 \pmod a $. We now study the target set
$$\{(y, z):\ c'y^2 + b'z^2 = (b' + c')p,\ b \mid (y^2 - 1),\ c \mid (z^2 - 1)\}.$$
By \cref{lem: solutions}, the only possible integer solutions up to sign are $(y,z) = (u - b'v, u + c'v)$ and $(y,z) = (u + b'v, u' - c'v)$. The first pair is actually a solution since
$$y^2 \equiv (u - b'v)^2 \equiv 1 \pmod{rq^w},$$
and
$$z^2 \equiv (u + c'v)^2 \equiv 1 \pmod{rq^w}.$$
As $p$ is not a square, the second pair must be a solution as well. This implies
$$(u + b'v)^2 \equiv 1 \pmod{b'd}.$$
However, we have that
$$(u + b'v)^2 \equiv \left(\frac{3b' - c'}{b' + c'}\right)^2 \pmod{q^w}.$$
As $w \geq v_q(b') + v_q(d)$ we conclude that
$$\left(\frac{3b' - c'}{b' + c'}\right)^2 \equiv 1 \pmod{q^{v_q(b') + h}}$$
which implies
$$8(b' - c') \equiv 0 \pmod{q^h}.$$
Thus we have $q^h \mid 24(b' - c')$ as desired.
\end{proof}
\begin{corollary}[\cref{thm:Classification}]
The only triples of positive integers $(a,b,c)$ such that $f_a\equiv f_bf_c\pmod 2$ are of the form $$(2q,4q,4q),(4q,8q,8q),$$ where $q$ is any positive odd number, or are members of the finite set: $$\{(4,6,12),(6,8,24),(8,12,24),(10,12,60),(15,24,40),(16,24,48),(20,24,120),(21,24,168)\}.$$
 \end{corollary}
 \begin{proof}
 By \cref{prop:b'c'} and \cref{prop:d}, we have $\frac{1}{a} = \frac{1}{b} + \frac{1}{c}$, $(b' + c') \mid 24$, and $d \mid 24(b' - c')$. Thus we have reduced all cases where $b' \neq c'$ to a finite computation, which we carried out to conclude the sporadic tuples above.
 
 For the case where $b' = c'$, since $b'$ and $c'$ are coprime, we have $b' = c' = 1$. The statement $f_a\equiv f_bf_c\pmod 2$ now reduces to 
 $$\abs{\{(y, z):\ y^2 + z^2 = 2\big(\frac{d}{2}x + 1\big),\ d \mid (y^2 - 1),\ d \mid (z^2 - 1)\}} \equiv 1 \pmod 2$$
 if and only if $\frac{d}{2}x + 1$ is a square. There is an involution on this set $(y,z) \to (z,y)$. Thus the parity of this set is equal to the number of fixed point, that is
 \begin{equation}
  \label{eq: cardinality} \abs{\{y:\ y^2 = \frac{d}{2}x + 1,\ d \mid (y^2 - 1)\}}   
 \end{equation}
 $$.$$
 If $\frac{d}{2}x + 1$ is not a square then \cref{eq: cardinality} is obviously zero. If $\frac{d}{2}x + 1$ is a square, then  \cref{eq: cardinality} is non-zero if and only if $\frac{d}{2}x + 1$ is equal to $1$ modulo $d$. This implies that under our assumptions, a square is equal to $1$ modulo $d$ if and only if it is equal to $1$ modulo $d / 2$. This is equivalent to $v_2(d) \in \{2, 3\}$, giving the tuples $(a,b,c) = (2q, 4q, 4q)$ and $(a,b,c) = (4q, 8q, 8q)$.
 \end{proof}

\end{document}